\documentclass[12pt, amscd]{amsart}
\usepackage{graphicx}
\usepackage{amsmath,amssymb, amsthm}
\usepackage[unicode]{hyperref}

\newtheorem{thm}{Theorem}[section]
\newtheorem{cor}[thm]{Corollary}
\newtheorem{lem}[thm]{Lemma}
\newtheorem{prop}[thm]{Proposition}
\theoremstyle{definition}
\newtheorem{defn}[thm]{Definition}

\newtheorem{rem}[thm]{Remark}
\DeclareMathOperator{\lk}{lk}
\DeclareMathOperator{\st}{st}
\DeclareMathOperator{\core}{core}

\DeclareMathOperator{\girth}{girth}

\begin{document}

\title[Cohen-Macaulay graphs with large girth] {Cohen-Macaulay graphs with large girth}
\author{D\^o Trong Hoang}
\address{Institute of Mathematics, Vietnam Academy of Science and Technology, 18 Hoang Quoc Viet, Hanoi, Vietnam}
\email{dotronghoang@gmail.com}
\author{Nguy\^en C\^ong Minh}
\address{Department of Mathematics, Hanoi National University of Education, 136 Xuan Thuy, Hanoi,
Vietnam}
\email{minhnc@hnue.edu.vn}
\author{Tr\^an Nam Trung}
\address{Institute of Mathematics, Vietnam Academy of Science and Technology, 18 Hoang Quoc Viet, Hanoi, Vietnam}
\email{tntrung@math.ac.vn}
\dedicatory{Dedicated to Professor Ng\^o Vi\^et Trung\\ on the occasion of his sixtieth birthday}
\date{}
\commby{}
\begin{abstract}
We classify Cohen-Macaulay graphs of girth at least $5$ and planar Gorenstein graphs of girth at least $4$. Moreover, such graphs are also vertex decomposable.
\end{abstract}
\maketitle
\section*{Introduction}

To each finite simple graph $G$ with the vertex set $V(G)\subseteq \{1, \ldots ,n\}$ and the edge set $E(G)$, one associates
the edge ideal $I(G)$ of the polynomial ring $R=k[x_1, \ldots , x_n]$ which is generated by all monomials $x_ix_j$ such that
$\{i, j\}\in E(G)$. Here $k$ is an arbitrary field. A graph $G$ is called {\it Cohen-Macaulay} (resp. {\it Gorenstein}) over
$k$, if $R/I(G)$ is a Cohen-Macaulay ring (resp. a Gorenstein ring).

It is a wide open problem to characterize graph-theoretically the class of Cohen-Macaulay graphs. However the Cohen-Macaulay property of graphs is characteristics dependent (see \cite[Exercise 5.3.31]{Vi}). Therefore the work on Cohen-Macaulay graphs now has focused on certain subclasses of graphs such as: chordal graphs, bipartite graphs and so on (see \cite{HH}, \cite{HHZ}, \cite{MKY},\cite{W1},\cite{W2}). Our aim of this paper is to classify Cohen-Macaulay graphs of girth at least $5$. Recall that the {\it girth} of a graph $G$, denoted by $\girth (G)$, is the length of any shortest cycle in $G$ or in the case $G$ is a forest we consider the girth to be infinite.

If $G$ is Cohen-Macaulay, then every maximal independent set of $G$ has the same size, namely $\alpha(G)$, the independence number of $G$ (see, e.g. \cite[Proposition $6.1.21$]{Vi}). Such a graph is called {\it well-covered} (see \cite{P1}). Thus we are naturally interested in characterizing well-covered graphs. However this problem turns out to be difficult (see \cite{P2}); and a striking result given in \cite{FHN1} is characterized only well-covered graphs of girth at least $5$.  This result plays an important role in our work to classify Cohen-Macaulay graphs of girth at least $5$ (see Theorem $\ref{T2}$).

We are next interested in Gorenstein graphs. As until now, we only know a classification of Gorenstein bipartite graphs (see \cite {HH}). If $G$ is a Gorenstein graph without isolated vertices, then $G$ is not only well-covered but also $G\setminus x$ is well-covered with $\alpha(G)=\alpha(G\setminus x)$ for any vertex $x$. Such a graph is called a member of $W_2$ (see \cite{Sp}). To characterize the class $W_2$ is also difficult (see \cite{P2}) and we only know a classification of triangle-free planar graphs in $W_2$ (see \cite{Pi1}). Moreover, if $G$ is such a graph, it is conjectured in \cite{RTY} that $I(G)^2$ is Cohen-Macaulay. In this paper, we will prove this conjecture (see Proposition \ref{EG1}). Using this result we are able to classify planar triangle-free Gorenstein graphs (see Theorem $\ref{T4}$).

The paper consists of three sections. In Section 1, we set up some basic notations, terminologies for the simplicial complex and the graph. In Section $2$, we classify Cohen-Macaulay graphs of girth at least $5$. In the last section, we classified planar Gorenstein graphs of girth at least $4$.

\section{Preliminaries}

Let $\Delta$ be a simplicial complex on $\{1,\ldots,n\}$. The Stanley-Reisner ideal of a simplicial complex $\Delta$ is a squarefree monomial ideal (see \cite{S}):
$$I_{\Delta} = (x_{j_1} \cdots x_{j_i} \mid j_1  <\cdots < j_i \ \text{ and } \{j_1,\ldots,j_i\} \notin \Delta)$$
of the polynomial ring  $R=k[x_1,\ldots,x_n]$, where $k$ is a field.

A $k$-algebra $k[\Delta] = R/I_{\Delta}$ is called the Stanley-Reisner ring of $\Delta$. We say that $\Delta$ is Cohen-Macaulay (resp. Gorenstein) (over $k$) if $k[\Delta]$ is Cohen-Macaulay (resp. Gorenstein). The dimension of a face $F\in\Delta$ is $\dim F = |F| -1$, where $|F|$ stands  for the cardinality of $F$, and the dimension of $\Delta$ is $\dim\Delta = \max\{\dim F \mid F \in \Delta\}$.

Our tool to study Cohen-Macaulayness of simplcial complexes is the notion of {\it vertex decomposible}. A simplicial complex $\Delta$ (not necessarily pure) is recursively defined to be vertex decomposable
if it is either a simplex or else has some vertex $v$ so that:
\begin{enumerate}
\item both $\Delta\setminus v$ and $\lk_{\Delta}v$ are vertex decomposable, and
\item no face of $\lk_\Delta(v)$ is a facet of $\Delta\setminus v$.
\end{enumerate}
Vertex decompositions were introduced in the pure case by Provan and Billera \cite{PB} and extended to non-pure complexes by Bj\"{o}rner and Wachs in \cite[Section 11]{BW}. It is well-known that if $\Delta$ is pure and vertex-decomposable then $\Delta$ is Cohen-Macaulay (see e.g. \cite{W1}).

We now recall some terminologies of graph theory. Two vertices $u, v$ of $G$ are adjacent if $uv$ is an edge of $G$. An independent set in $G$ is a set of vertices no two of which are adjacent to each other. An independent set of maximum size will be referred to as a maximum independent set of $G$, and the independence number of $G$, denoted by $\alpha(G)$, is the cardinality of a maximum independent set in $G$. An independent set $S$ in $G$ is maximal (with respect to set inclusion) if the addition to $S$ of any other vertex in the graph destroys the independence.

Let $\Delta(G)$ be the set of all independent sets in $G$. Then $\Delta(G)$ is a simplicial complex and is so-called the independence complex of $G$. Note that $I(G) = I_{\Delta(G)}$ and $\dim(\Delta(G)) = \alpha(G)-1$. Clearly, $G$ is well-covered if and only if $\Delta(G)$ is pure. A graph $G$ is called a vertex decomposable graph if so is its independence complex.

If $X \subseteq V(G)$, then $G[X]$ is the subgraph of $G$ spanned by $X$. By $G\setminus W$, we mean the induced subgraph $G[V \setminus W]$ for some $W \subseteq V(G)$. The neighborhood of a vertex $v$ of $G$ is the set $N_G(v) = \{w \mid w \in V(G) \text{ and } vw\in E(G)\}$, and let $N_G[v] = N_G(v) \cup \{v\}$; if there is no ambiguity on $G$, we use $N(v)$ and $N[v]$, respectively. Let $G_v=G\setminus N_G[v]$.

The vertex decomposability of $G$ can be interpreted in terms of $G$ itself by an observation (\cite{W1}[Lemma 4]) as follows.

A graph $G$ is vertex decomposable if $G$ is a totally disconnected graph (with no edges) or if it has some vertex $v$ so that:
\begin{enumerate}
\item $G\setminus v$ and $G_v$ are both vertex decomposable, and
\item no independent set in $G_v$ is a maximal independent set in $G\setminus v$.
\end{enumerate}

In order to study the vertex decomposability of a graph, it suffices to consider its connected graphs.

\begin{lem}[W1, Lemma 20]\label{lemCM} $G$  is vertex decomposable if and only if all its connected components are vertex decomposable.
\end{lem}


\section{Cohen-Macaulayness versus Vertex decomposablity}

In this section we will classify Cohen-Macaulay graphs of girth at least $5$. First we recall a class $\mathcal{SQC}$ of well-covered graphs from \cite{RV}. This class is of interest since, as we will see, it contains all Cohen-Macaulay graphs of girth at least $5$.

A vertex $v$ of a graph $G$ is said to be {\it simplicial} if the induced subgraph of $G$ on the set $N[v]$ is a complete graph and we say this complete graph to be a simplex of $G$. A $5$-cycle $C_5$ of a graph $G$ is called {\it basic} if $C_5$ does not contain two adjacent vertices of degree three or more in $G$; a $4$-cycle $C_4$ is called {\it basic} if it contains two adjacent vertices of degree two, and the remaining two vertices belong to a complete subgraph or a basic $5$-cycle of $G$. A graph $G$ is in the class $\mathcal{SQC}$ if there are simplicial vertices $x_1,\ldots,x_m$; basic $5$-cycles $C^1,\ldots,C^s$; and basic $4$-cycles $Q^1,\ldots,Q^t$ such that
$$V(G) = \bigcup_{j=1}^m N[x_j] \cup \bigcup_{j=1}^s V(C^j) \cup \bigcup_{j=1}^t B(Q^j)$$
and this forms a partition of $V(G)$, where $B(Q^j)$ is the set of two vertices of degree $2$ of the basic $4$-cycle $Q^j$. Such the graph $G$ is well-covered \cite[Theorem $3.1$]{RV}. Moreover, from the proof of this result, we also have a formula to compute the independence number of such graphs:
\begin{equation}\label{EQ1}\alpha(G) = m + 2s + t.\end{equation}

The first main result of this section says that all the graphs $G$ in the class $\mathcal{SQC}$ are vertex decomposable. The proof is divided into a number of steps. First, we deal with well-covered simplicial graphs. A graph $G$ is said to be simplicial if every vertex of $G$ belongs to a simplex of $G$. Using a characterization due to Prisner, Topp and Vestergaard in \cite[Lemma 2]{PTV} of such graphs, we see that all well-covered simplicial graphs belong to the class $\mathcal{SQC}$. Moreover,

\begin{lem}[W2, Corollary 5.5]\label{L03} If $G$ is a (well-covered) simplicial graph, then $G$ is vertex decomposable.
\end{lem}

Next, we deal with graphs in the class $\mathcal{SC}$. A graph $G$ is called in the class $\mathcal{SC}$ if $V(G)$ can be partitioned into two disjoint subsets $S$ and $C$: the subset $S$ contains all vertices of the simplexes of $G$, and the simplexes of $G$ are vertex disjoint; the subset $C$ consists of the vertices of the basic $5$-cycles and the basic $5$-cycles form a partition of $C$. Obviously, the class $\mathcal{SC}$ is a subclass of the class  $\mathcal{SQC}$.

\begin{lem}\label{L04} If $G$ is a graph in the class $\mathcal{SC}$, then $G$ is vertex decomposable.
\end{lem}
\begin{proof} We prove by induction on $|V(G)|$. If $|V(G)| < 5$, then $G$ is simplicial. Therefore the lemma follows from Lemma $\ref{L03}$.

Assume that $|V(G)|\geq 5$. If $G$ is disconnected, let $G_1,\ldots, G_m$ be components of $G$. Note that each $G_i$ is also in the class $\mathcal{SC}$. Since $|V(G_i)| < |V(G)|$, by the induction hypothesis, $G_i$ is vertex decomposable. Thus, $G$ is also vertex decomposable by \cite[Lemma 20]{W1}.

Assume that $G$ is connected.  Let $C^1,\ldots,C^s$ be basic $5$-cycles and $x_1,\ldots,x_t$ simplicial vertices of $G$ such that
$$V(C^1), \ldots, V(C^s), N[x_1], \ldots, N[x_t]$$
form a partition of $V(G)$. If $s = 0$, then the lemma follows from Lemma $\ref{L03}$. So we may assume that $s \geq 1$.
Write $C^1 =\{xy, yz,zu,uv,vx\}$ with $\deg_G(x)\geqslant 3$.

We first claim that $G\setminus x$ is vertex decomposable. Let $H = G\setminus x$. Since $C^1$ is a basic $5$-cycle of $G$, we imply that $\deg_H(y) = \deg_H(v) = 1$. Therefore, $C^2,\ldots,C^s$ are also basic $5$-cycles of $H$ and $x_1,\ldots,x_t,y,v$ are simplicial vertices of $H$. Clearly,
$$V(H) = V(C^2)\cup \cdots\cup V(C^s)\cup N_H[x_1]\cup\cdots\cup N_H[x_t]\cup N_H[y] \cup N_H[v]$$
and this is a partition of $V(H)$. In particular, $H$ belongs to $\mathcal{SC}$. Furthermore, $|V(H)| = |V(G)|-1$, so by induction we have $H$ is vertex decomposable, as claimed.

We next claim that $G_x$ is vertex decomposable. Let $L = G_x$. Since $C^1$ is a basic $5$-cycle, either $z$ or $u$ has degree $2$. Assume that $\deg_G(z)=2$, so that $z$ is a simplicial vertex of $L$. Without loss of generality, we may assume that $C^2,\ldots, C^m$ are all basic $5$-cycles which have vertices being adjacent to $x$. Observe that $C^{m+1}, \ldots,C^s$ are basic $5$-cycles of $L$ and $x_1,\ldots,x_t$ are simplicial vertices of $L$. For each $i=2,\ldots, m$, let $c_i$ be a vertex of $C^i$ that is adjacent to $x$ in $G$; and let $u_i$ and $v_i$ be two adjacent vertices of $c_i$ in the cycle $C^i$. Since $u_i$ and $v_i$ are of degree $2$ in $G$, we then have they are two simplicial vertices of $L$  and
$$V(C^{m+1}),\ldots,V(C^s),N_L[z], N_L[u_2], N_L[v_2], \ldots, N_L[u_m], N_L[v_m], N_L[x_1],\ldots,N_L[x_t]$$
form a partition of $V(L)$. Which implies that $L$ is in the class $\mathcal{SC}$. Since $|V(L)| < |V(G)|$, by induction, we have $L$ is vertex decomposable, as claimed.

We now turn to prove the lemma. Since $\alpha(H) = 2(s-1) + (t+2) = 2s+t$ and $\alpha(L) =2(s-m)+1+2(m-1)+t = 2s+t-1=\alpha(H)-1$, together two claims, we yield $G$ is vertex decomposable, as required.
\end{proof}

We now in position to prove that every member of $\mathcal{SQC}$ is vertex decomposable.

\begin{thm}\label{T1} If $G$ is a graph in the class $\mathcal{SQC}$, then $G$ is vertex decomposable. In particular, this graph is Cohen-Macaulay.
\end{thm}
\begin{proof} We prove by induction on $|V(G)|$. If $|V(G)| < 3$, then $G$ is a well-covered simplicial graph. Hence, $G$ is vertex decomposable by Lemma $\ref{L03}$.

Assume that $|V(G)|\geq 3$. Let $C^1,\ldots,C^s$ be basic $5$-cycles; $x_1,\ldots,x_t$ simplicial vertices; and $Q^1,\ldots,Q^m$ basic $4$-cycles of $G$ such that
$$V(G) =  \bigcup_{j=1}^t N[x_j] \cup \bigcup_{j=1}^s V(C^j) \cup \bigcup_{j=1}^m B(Q^j)$$
and this is a partition of $V(G)$, where $B(Q^j)$ is the set of two vertices of degree $2$ of the basic 4-cycle $Q^j$ for $j=1,\ldots,m$. If $m=0$, then $G$ is in the class $\mathcal{SC}$, and then $G$ is vertex decomposable by Lemma $\ref{L04}$.

If $m\geq 1$. Let $c$ be a vertex in a basic 4-cycle $Q^1=\{a_1b_1, b_1c, cd_1,d_1 a_1\}$ with $ \deg_G(c)\geq 3$. Write $\deg_G(a_1)=\deg_G(b_1) =2$ and $\deg_G(d_1)\geq 3.$
Without loss of generality, we may assume that $c\in V(Q^i)$ for $i=1,\ldots,l$ and $c\notin V(Q^i)$ for $i=l+1,\ldots,m$. We can write $Q^i =\{a_ib_i,b_ic,cd_i,d_ia_i\}$ with $\deg_G(a_i)=\deg_G(b_i) = 2$ for  $i=2,\ldots,l$. Note that $a_1, b_1,\ldots,a_l,b_l$ are distinct points, but $d_1,\ldots,d_l$ may be not distinct points.

We now distinguish two cases:

\textbf {Case 1:} $c$ lies in some basic $5$-cycle of $G$. Since $G$ is in $\mathcal{SQC}$, $c$ lies in only one basic $5$-cycle. We may assume that $c$ lies in $C^1$ and $C^1 =\{cu_1,u_1y_1,y_1z_1,z_1v_1,v_1c\}$. Since $\deg_G(c)\geqslant 3$, $\deg_G(u_1)=\deg_G(v_1)=2$.

We first claim that $H=G\setminus c$ is vertex decomposable. Indeed, since $\deg_H(b_1)=\cdots=\deg_H(b_l)=1$, $b_1,\ldots,b_l$ are simplicial vertices of $H$. It is easy to check that $u_1,v_1,b_1,\ldots,b_l, x_1,\ldots, x_t$ are all simplicial vertices; $C^2,\ldots,C^s$ are basic $5$-cycles; and $Q^{l+1},\ldots,Q^m$ are basic $4$-cycles of $H$. Moreover,
$$V(H) = N_H[u_1] \cup N_H[v_1]\cup \bigcup_{j=1}^l N_H[b_j] \cup \bigcup_{j=1}^t N_H[x_j]\cup \bigcup_{j=2}^{s} V(C^j) \cup \bigcup_{j=l+1}^m B(Q^j)$$
and this is a partition of $V(H)$. Thus, $H$ is in the class $\mathcal{SQC}$ and $|V(H)|=|V(G)|-1$. By induction, $H$ is vertex decomposable, as claimed.

Moreover, by Formula $(\ref{EQ1})$ we get $\alpha(H) = 1 + 1 + l + t + 2(s-1) + (m-l) = t+2s+m$.

We claim further that $L=G_{c}$ is also vertex decomposable. Indeed, it is clear that $a_1,\ldots, a_l$ are isolated vertices of $L$. Therefore, they are simplical vertices of $L$. Since $C^1$ is a basic $5$-cycle, either $y_1$ or $z_1$ has degree 2 in $G$. By symmetry, we can assume that assume $\deg_G(y_1)=2$. Then, $\deg_L(y_1)\leq 1$ , and then $y_1$ is a simplicial vertex of $L$. We can assume that each of $Q^{l+1},\ldots,Q^{l+r}$ has at least one vertex being adjacent to $c$; and every $Q^{l+r+1},\ldots,Q^m$ has no any vertices being adjacent to $c$. Write $Q^{j} = \{a_jb_j, b_jc_j,c_jd_j,d_ja_j\}$ with $\deg_G(a_j)=\deg_G(b_j)=2$ and $c$ is adjacient with $c_j$ for all $j=l+1,\ldots,l+r$. Hence, $b_{l+1},\ldots,b_{l+r}$ are simplicial in $L$.

We also can assume that each of $C^2,\ldots,C^p$ has at least one vertex being adjacient with $c$; and every $C^{p+1},\ldots,C^s$ has no any vertices being adjacent to $c$. For each $i = 2,\ldots, p$, let $C^i =\{u_iy_i,y_iz_i,z_iv_i,v_iw_i,w_iu_i\}$ with $c$ and $w_i$ are adjacent in $G$. So, $\deg_G(u_i) = \deg_G(v_i)=2$. Hence, both of $u_i$ and $v_i$ are simplicial in $L$.

Since $G$ is a member of the class $\mathcal{SQC}$, we conclude that $c\notin N_G[x_i]$ for all $i=1,\ldots, t$. Thus, $x_i$ is also simplicial in $L$ for all $i$.

In summary, $L$ has simplicial vertices
$$y_1, a_1,\ldots,a_l, b_{l+1},\ldots, b_{l+r}, u_2,v_2,\ldots, u_p,v_p, x_1,\ldots,x_t;$$
basic $5$-cycles $C^{p+1},\ldots,C^{s-1}$; and basic $4$-cycles $Q^{l+r+1},\ldots,Q^m$. Moreover,
\begin{align*}
V(L) = N_L[y_1]&\cup \bigcup_{j=1}^l N_L[a_j] \cup \bigcup_{j=1}^r N_L[b_{l+j}] \cup \bigcup_{j=2}^p \left(N_L[u_j] \cup N_L[v_j]\right)\\
&\cup \bigcup_{j=1}^t N_L[x_j]\cup \bigcup_{j=p+1}^s V(C^j) \cup \bigcup_{j=l+r+1}^m B(Q^j),
\end{align*}
and this is a partition of $V(L)$. Therefore, $L$ is in the class $\mathcal{SQC}$. By Formula ($\ref{EQ1}$),
$$\alpha(L) = 1 + l +r +2(p-1) +t+ 2(s- p ) + (m-l-r) = t + 2s +m-1 = \alpha(H)-1.$$
Since $|V(L)|<|V(G)|$, we have $L$ is vertex decomposable by induction. Then, $G$ is vertex decomposable.

\textbf {Case 2:} $c$ does not lies in any basic $5$-cycle of $G$. Then, $c$ belongs to only one of the simplices $N_G[x_1],\ldots,N_G[x_t]$. In the same way as the proof of Case $1$, we have $G$ is vertex decomposable. 

By \cite[Theorem $3.1$]{RV}, any graph in the class $\mathcal{SQC}$ is always well-covered. Then, $G$ is Cohen-Macaulay as required.
\end{proof}

An edge, in a graph $G$, incident with a point of degree 1 is called the {\it pendant} edge. Let $C(G)$ denote the set of all vertices which belong to basic $5$-cycles and let $P(G)$ denote the set of vertices  which are incident with pendant edges in $G$. Then, $G$ is in the class $\mathcal{PC}$ if $V(G)$ can be partitioned into $V(G)=P(G) \cup C(G)$ and the pendant edges form a perfect matching of $P(G)$. If $uv$ is a pendant edge in $G$ with $\deg(u)=1$, then $N[u]=\{u,v\}$, and then $u$ is a simplicial vertex in $G$. Hence, $\mathcal{PC}$ is a subclass of $\mathcal{SQC}$.

We are now ready to prove the main result of this paper.

\begin{thm}\label{T2} Let $G$ be a connected graph of girth at least $5$. Then, the following statements are equivalent:
\begin{enumerate}
\item $G$ is well-covered and vertex decomposable;
\item $G$ is Cohen-Macaulay;
\item $G$ is either a vertex or in the class $\mathcal{PC}$.
\item $G$ is in the class $\mathcal{SC}$.
\item $G$ is in the class $\mathcal{SQC}$.
\end{enumerate}
\end{thm}
\begin{proof} (1)$\Longrightarrow$(2) is well known.

(2)$\Longrightarrow$(3):\quad If $G$ is a Cohen-Macaulay graph, then $G$ is well-covered. By \cite{FHN1}, we have either $G$ is in the class $\mathcal{PC}$ or $G$ is one of six exceptional graphs shown in Figure $1$. Among these six exceptional graphs, only $K_1$ is Cohen-Macaulay (see \cite[Proposition 3.3]{Br}). Thus, $G$ is either a vertex or in the class $\mathcal{PC}$.

(3)$\Longrightarrow$(4) and (4)$\Longrightarrow$(5) hold true by definition of the classes $\mathcal{PC}, \mathcal{SC}$ and$\mathcal{SQC}$.

(5)$\Longrightarrow$(1) are done by Theorem \ref{T1}.

\begin{figure}[ht!]
\begin{minipage}{.30\textwidth}
\begin{center}
\scalebox{0.5}{\includegraphics{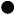}}\\
$K_1$
\end{center}
\end{minipage}
\begin{minipage}{.30\textwidth}
\begin{center}
\scalebox{0.5}{\includegraphics{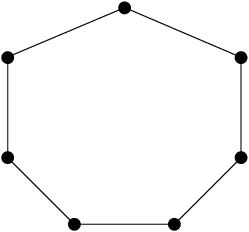}}\\
$C_7$
\end{center}
\end{minipage}
\begin{minipage}{.30\textwidth}
\begin{center}
\scalebox{0.5}{\includegraphics{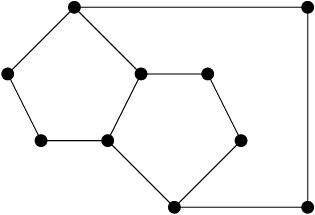}}\\
$P_{10}$
\end{center}
 \end{minipage}

 \vspace{1cm}
\begin{minipage}{.30\textwidth}
\begin{center}
\scalebox{0.5}{\includegraphics{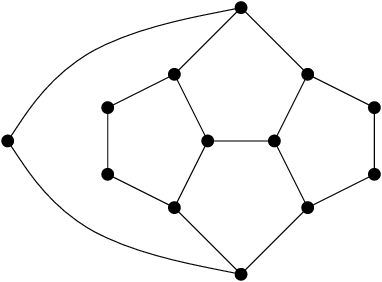}}\\
$P_{13}$
\end{center}
\end{minipage}
\begin{minipage}{.30\textwidth}
\begin{center}
\scalebox{0.5}{\includegraphics{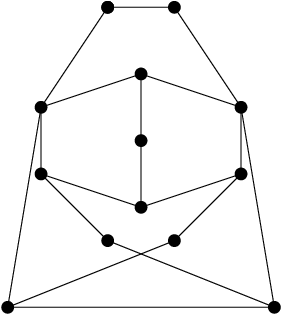}}\\
$Q_{13}$
\end{center}
\end{minipage}
\begin{minipage}{.30\textwidth}
 \begin{center}
\scalebox{0.5}{\includegraphics{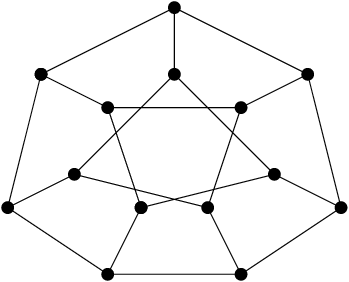}}\\
$P_{14}$
\end{center}
\end{minipage}
\caption{}\label{Fig:1}
\end{figure}


\end{proof}

The following corollary is immediate (also see \cite[Corollary $4.3$]{P2}).

\begin{cor} \label{cor3} Let $G\ne K_1$ be a connected graph of girth at least $6$. Then, $G$ is Cohen-Macaulay if and only if its pendant edges form a perfect matching of $G$.
\end{cor}

We conclude this section with characterizing some special classes of Cohen-Macaulay graphs in which triangles are allowed. First we consider block-cactus graphs. A vertex $v$ of a graph $G$ is called a cut vertex of $G$ if $G \setminus v$ has more components than $G$. A connected graph with no cut vertex is called a block. A block of a graph $G$ is a subgraph of $G$ which is itself a block and which is maximal with respect to that property. A graph $G$ is called a block-cactus graph if every block is complete or a cycle. We have a characterization of Cohen-Macaulay block-cactus graphs as follows.

\begin{cor}\label{cor2} Let $G$ be a block-cactus graph. Then the following statements are equivalent:
\begin{enumerate}
\item $G$ is well-covered and vertex decomposable.
\item $G$ is Cohen-Macaulay.
\item $G$ is in the class $\mathcal{SQC}$.
\end{enumerate}
\end{cor}

\begin{proof} (1)$\Longrightarrow$(2): obviously.

(2)$\Longrightarrow$(3):\quad It suffices to prove for connected block-cactus graphs. Since $G$ is Cohen-Macaulay, $G$ is well-covered. By \cite[Theorem 3.2]{RV}, $G$ belongs to the class $\{ \text{4-cycle, 7-cycle} \} \cup \mathcal{SQC}$. Since both of 4-cycle and 7-cycle are not Cohen-Macaulay, $G$ is in the class $\mathcal{SQC}$.

(3)$\Longrightarrow$(1):\quad follows from Theorem \ref{T1}.
\end{proof}

If every block of a connected block-cactus graph $G$ is an edge or a cycle, then $G$ is called a cactus graph. Equivalently, $G$ is a cactus graph if and only if it is connected and two cycles have at most one vertex in common. A $3$-cycle in $G$ is called basic if it contains at least one vertex of degree $2$. Corollary $\ref{cor2}$ can restate more explicitly in a combinatorial way for Cohen-Macaulay cactus graphs as follows (see \cite{MKY}).

\begin{cor}\label{cor3} Let $G$ be a cactus graph. Then the following statements are equivalent:
\begin{enumerate}
\item $G$ is well-covered and vertex decomposable.
\item $G$ is Cohen-Macaulay.
\item $G$ satisfies two following conditions:
\begin{enumerate}
\item every vertex of degree 2 is incident with only one pendant edge or one basic 3-cycle or one basic 4-cycle or one basic 5-cycle;
\item every vertex of degree at least 3 is incident with only one pendant edge or one basic 3-cycle or one basic 5-cycle.
\end{enumerate}
\end{enumerate}
\end{cor}

Finally, we will characterize Cohen-Macaulay graphs containing no $4$- nor $5$-cycles as subgraphs. In particular, in such graphs no cliques of size greater than $3$ can exist.

\begin{cor}\label{T3} Let $G$ be a graph that contains neither $4$-cycles nor $5$-cycles as subgraphs. Then, the following conditions are equivalent:
\begin{enumerate}
\item[(1)] $G$ is Cohen-Macaulay.
\item[(2)] There are simplicial vertices $x_1,\ldots,x_m$ of $G$ such that $\deg_G(x_i) \leq 3$ for all $i$ and $N_G[x_1],\ldots,N_G[x_m]$ form a partition of $V(G)$.
\item[(3)] $G$ is a well-covered simplicial graph such that every simplicial vertex has degree at most $3$.
\end{enumerate}
\end{cor}
\begin{proof} $(1) \Longrightarrow (2)$ Since $G$ is a well-covered graph containing no $4$-cycles nor $5$-cycles, by \cite[Theorem 1.1]{FHN2}, $G$ is either a well-covered simplicial graph such that every simplicial vertex has degree at most $3$ or one of two exceptional graphs $C_7$ and $T_{10}$ shown in Figure \ref{Fig:2}. But both of $C_7$ and $T_{10}$ are not Cohen-Macaulay, so $G$ satisfies the condition as in the second statement.

\begin{figure}[ht!]
\begin{minipage}{.30\textwidth}
\begin{center}
\scalebox{0.5}{\includegraphics{c7}}\\
$C_7$
\end{center}
\end{minipage}
\begin{minipage}{.30\textwidth}
\begin{center}
\scalebox{0.5}{\includegraphics{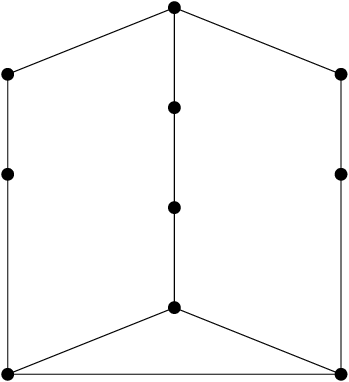}}\\
$T_{10}$
\end{center}
\end{minipage}
\caption{}\label{Fig:2}
\end{figure}

$(2) \Longleftrightarrow (3)$ and $(3) \Longrightarrow (1)$ hold true by \cite[Theorem 1]{PTV} and Lemma \ref{L03}.
\end{proof}

\section{Planar Gorenstein graphs of girth at least $4$}

In this section we characterize planar Gorenstein graphs of girth at least $4$. Recall that $W_2$ is the class of well-covered graphs $G$ such that $G\setminus x$ are well-covered with $\alpha(G)=\alpha(G\setminus x)$ for all vertices $x$. 

 First, we recall that a simplicial complex $\Delta$ is called {\it doubly Cohen-Macaulay} (over $k$) if $\Delta$ is Cohen-Macaulay (over $k$) and for every vertex $x$ of $\Delta$ the subcomplex $\Delta\setminus x$ is also Cohen-Macaulay (over $k$) of the same dimension as $\Delta$ (see \cite{B}). 
The restriction of $\Delta$ to a subset $W$ of the vertices set $V(\Delta)$ is $\Delta|_{W}=\{F\in\Delta~|~ F\subseteq W\}$. The star of a vertex $v$ in $\Delta$ is $\st_{\Delta}(v)=\{F\in\Delta~|~F\cup \{v\}\in\Delta\}$. Let $\core(\Delta)=\Delta|_{\{v\in V(\Delta)~|~\st_\Delta(v)\ne V(\Delta)\}}$. It is well known that if $\Delta$ is Gorenstein with $\core(\Delta)=\Delta$, then $\Delta$ is doubly Cohen-Macaulay (see \cite[Theorem II.5.1]{S}).

We have a necessary condition for Gorensteinness of graphs as follows.
\begin{lem}\label{G} Let $G$ be a Gorenstein graph without isolated vertices. Then, $G$ is a member of $W_2$.
\end{lem}
\begin{proof} Since $G$ is a Gorenstein graph without isolated vertices, $\core(\Delta(G)) = \Delta(G)$. Therefore, $\Delta(G)$ is doubly Cohen-Macaulay. It follows that for any vertex $x$ of $G$, we have $\Delta(G)\setminus x = \Delta(G\setminus x)$ is Cohen-Macaulay with
$$\alpha(G\setminus x) =\dim \Delta(G\setminus x) +1 = \dim \Delta(G)+1 = \alpha(G).$$
Thus, $G$ is in class $W_2$.
\end{proof}

Pinter \cite{Pi1} constructed an infinite family of graphs by a recursive procedure as follows:
\begin{enumerate}
\item Begin with the graph $G_3$ shown in Figure \ref{Fig:3};
\item Given any graph $G$ in the construction, let $x$ and $y$ be two adjacent points of degree $2$ in $G$. Let $u$ be the neighbor of $x$ such that $u \ne y$. Then construct a new graph $G'$ with precisely three more points than $G$ as follows. Let the three new points be $a, b$ and $c$. Now join $a$ to $x$ and $b$, $b$ to $c$ and $c$ to $u$ and $y$ (see Figure \ref{Fig:4}).
 \end{enumerate}

 \begin{figure}[!hb]
\begin{minipage}{5cm}
\scalebox{0.5}{\includegraphics{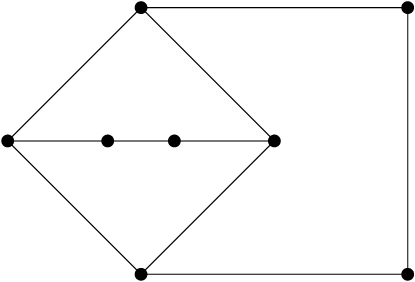}}
\caption{}\label{Fig:3}
\end{minipage}
\begin{minipage}{.30\textwidth}
\hspace{1em}\vspace{1.9ex}
 \scalebox{0.55}{ \includegraphics{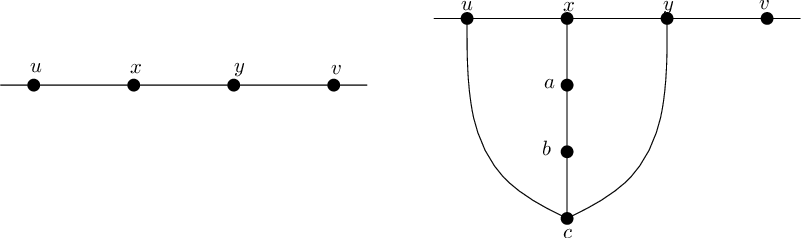}}
\caption{}\label{Fig:4}
\end{minipage}
\end{figure}

Let us denote this family by $\mathcal G$. Pinter proved the following result (see \cite{Pi1}).

\begin{lem} \label{P} A connected graph $G$ is a girth $4$ planar member of class $W_2$ if and only if $G$ is a member of the family $\mathcal G$.
\end{lem}

In general, two graphs $G$ and $H$ are isomorphic, written $G\cong H$, if there is a bijection map $\varphi: V(G)\longrightarrow V(H)$ such that $uv\in E(G)\Longleftrightarrow \varphi(u)\varphi(v)\in E(H)$ for all $u, v\in V(G)$. Thus, $G$ and $H$ can be identified if they are isomorphic.

\begin{defn} For every integer $n\geq 1$, we define two graphs $G_n$ and $H_n$ as follows:
\begin{enumerate}
\item $G_n$ is a graph with the vertex set $\{x_1,\ldots,x_{3n-1}\}$ and the edge set
$$\{x_1x_2, \{x_{3k-1}x_{3k}, x_{3k}x_{3k+1}, x_{3k+1}x_{3k+2}, x_{3k+2}x_{3k-2}\}_{k=1,2,\ldots,n-1}, \{x_{3l-3}x_{3l}\}_{l=2,3,\ldots,n-1}\}$$

\begin{figure}[ht!]
\begin{center}
 \includegraphics[scale=0.7]{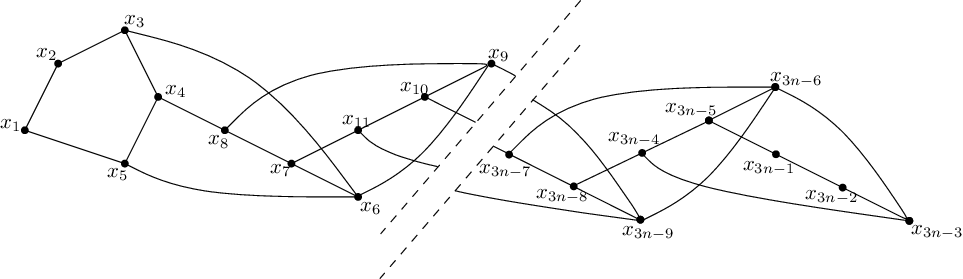}\\
  \caption[Figure 5.]{The graph $G_n$.}\label{Fig:5}
\end{center}
\end{figure}
\item $H_n =G_n \setminus x_{3n-1}$. It means that $V(H_n)=\{x_1,\ldots,x_{3n-2}\}$ and
\begin{align*}
E(H_n) =
\begin{cases}
\{x_1\} &\text{ if } n=1,\\
\{x_1x_2, x_2x_3,x_3x_4\} &\text{ if } n=2,\\
E(G_{n-1}) \cup \{x_{3n-2}x_{3n-3}, x_{3n-3}x_{3n-4}, x_{3n-3}x_{3n-6}\} &\text{ if } n\geq 3.
\end{cases}
\end{align*}

\begin{figure}[ht!]
\begin{center}
 \includegraphics[scale=0.7]{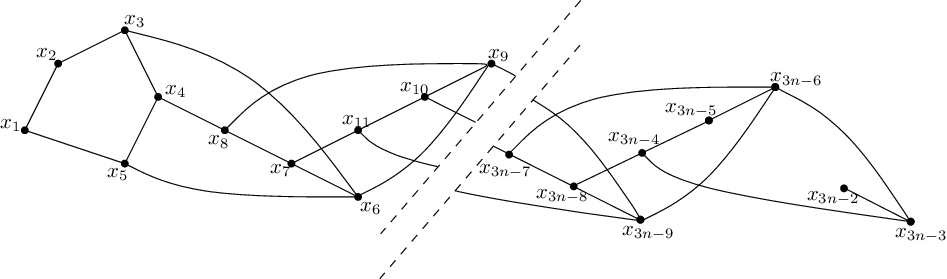}\\
  \caption[Figure 6.]{The graph $H_n$.}\label{Fig:6}
\end{center}
\end{figure}
\end{enumerate}
\end{defn}

In fact, from the construction of $\mathcal G$ and Lemma \ref{P}, we will obtain the following.

\begin{lem} A connected planar graph $G$ of girth $4$ is a member of $W_2$ if and only if $G\cong G_n$ for some $n\geq 3$.
\end{lem}
\begin{proof}
We claim that a graph $L\in \mathcal G$ if and only if $L\cong G_n$ for some $n\geq 3$.

By the construction, $G_3\in \mathcal G$. We will reconstruct $G_{n+1}$ from $G_n$ (for $n\ge 3$) as in the definition of $\mathcal G$ by $x=x_{3n-2}; y=x_{3n-1}$ and the three new points $a=x_{3n+2}, b=x_{3n+1}$ and $c=x_{3n}$. It implies that $G_{n+1}\in \mathcal G$. 

Conversely, assume $L\in \mathcal G$ and $L\cong G_n$ for some $n\geq 3$. It is easy that there are only or $x_1$ and $x_2$ two adjacent points of degree $2$ or  $x_{3n-2}$ and $x_{3n-1}$ two adjacent points of degree $2$ in $G$ (by induction on $n\geq 3$). By rewrite the order of vertices, we can take $x=x_{3n-2}$ and $y=x_{3n-1}$; set three new points $a=x_{3n+2}, b=x_{3n+1}$ and $c=x_{3n}$ as in the construction of $\mathcal G$. It is clear that the new graph and $G_{n+1}$ is isomorphic. 

Using this claim and Lemma \ref{P}, our assertion is proved.
\end{proof}

Note that a graph is Gorenstein if and only if every its component is Gorenstein. Thus, it suffices to characterize connected planar Gorenstein graphs. Let $G$ be a connected graph in the class $W_2$.  Pinter \cite{Pi2} proved that if $G\ne K_2$ or $C_5$ then $\girth(G)\leq 4$. Thus, connected Gorenstein graphs with girth at least $5$ is one of three graphs $K_1, K_2$ and $C_5$.  So the structure of connected Gorenstein graphs is non trivial only for the ones of girth $3$ or $4$. In the last theorem, we will  give a complete characterization of a Gorenstein connected planar graph of girth $4$. Recall that the union of graphs $G$ and $H$ is the graph $G \cup H$ with vertex set $V(G) \cup V (H)$ and edge set $E(G) \cup E(H)$. If $G$ and $H$ are disjoint, we refer to their union as a disjoint union, and generally denote it by $G \sqcup H$. Firstly, we have the following observation.

\begin{rem}\label{L01} Let $G$ be a graph and a point $x\in V(G)$. If both of $G_x$  and $G\setminus x$ are well-covered with $\alpha(G\setminus x)=\alpha(G_x)+1$, then $G$ is also well-covered with $\alpha(G)=\alpha(G\setminus x)$.
\end{rem}

Now we prove the vertex decomposability of $G_n$ and $H_n$.

\begin{lem} \label{faCM} For all integers $n\geq 1$, both of $G_n$ and $H_n$ are well-covered and vertex decomposable with $\alpha(G_n)=\alpha(H_n)=n$. In particular, $G_n$ and $H_n$ are Cohen-Macaulay.
\end{lem}
\begin{proof} We will prove by induction on $n$. If $n=1$ or $n=2$, then the lemma obviously holds true. If $n\geq 3$, since $H_n \setminus x_{3n-3} = G_{n-1} \sqcup  \{x_{3n-2}\}$, so $H_n \setminus x_{3n-3}$ is well-covered and vertex decomposable with $\alpha(H_n \setminus x_{3n-3}) = \alpha(G_{n-1}) + 1 = n$ by induction and Lemma $\ref{lemCM}$.
On the other hand, it is clear that $(H_n)_{x_{3n-3}} = G_{n-2} \sqcup \{x_{3n-5}\}$. Also by induction and Lemma $\ref{lemCM}$, $(H_n)_{x_{3n-3}}$ is well-covered and vertex decomposable with $\alpha((H_n)_{x_{3n-3}}) = \alpha(G_{n-2}) + 1 = n-1 = \alpha(H_n \setminus x_{3n-3}) - 1$. Therefore, $H_n$ is well-covered with $\alpha(H_n) = \alpha(H_n \setminus x_{3n-3}) = n$ (by Remark $\ref{L01}$) and vertex decomposable (by Lemma $\ref{lemCM}$).

Moreover, $G_n \setminus x_{3n-1} = H_n$ is well-covered and vertex decomposable with $\alpha(G_n \setminus x_{3n-1}) = n$ has done.  Let the map $\varphi: V(G_{n-1}) \longrightarrow V((G_n)_{x_{3n-1}})$ by $\varphi(x_i)=x_i$ for all $1\le i\le 3n-6$ or $i=3n-4$ and $\varphi(x_{3n-5})=x_{3n-3}$. It is clear that $\varphi$ is an isomorphism of two graphs $G_{n-1}$ and $(G_n)_{x_{3n-1}}$. Then, by induction, $(G_n)_{x_{3n-1}}$ is vertex decomposable and $\alpha((G_n)_{x_{3n-1}}) = n-1$. Thus, $G_n$ is vertex decomposable by  Lemma $\ref{lemCM}$. It is clear that $G_n$ is well-covered with $\alpha(G_n)=n$, which is complete the proof.
\end{proof}

Next, we prove that $I(G_n)^2$ is Cohen-Macaulay for all integers $n\geq 1$. This settles a conjecture of G. Rinaldo, N. Terai and K. Yoshida \cite[Conjecture 5.7]{RTY}. The case $n=1$ is known in \cite[Theorem 3.2]{MT} and the case $n=2$ is also mentioned in \cite[Theorem 3.8 (iv)]{TrT}.

\begin{prop} \label{EG1} $I(G_n)^2$ are Cohen-Macaulay for all integers $n\geq 1$.
\end{prop}

\begin{proof} Note that $G_n$ is a triangle-free graph, so $I(G_n)^2 = I(G_n)^{(2)}$ (see e.g. \cite[Corollary 4.5]{RTY}). Therefore, it suffices to prove that $I(G_n)^{(2)}$ is Cohen-Macaulay.

If $n=1$ (resp. $n=2$) then $G_n$ is an edge (resp. a pentagon). Thus $I(G_n)^{(2)}$ is Cohen-Macaulay.

If $n\geq 3$, by Lemma $\ref{faCM}$ and \cite[Theorem 2.3]{HMT}, it is enough to prove that $(G_n)_{xy}$ is Cohen-Macaulay with $\alpha((G_n)_{xy}) = n-1$ for every edge $xy\in E(G_n)$; where $(G_n)_{xy}$ stands for $G_n\setminus (N_{G_n}(x) \cup N_{G_n}(y))$. We distinguish six following cases:

\textbf {Case 1:}  $xy = x_1x_2$. Clearly, $(G_n)_{x_1x_2} \cong H_{n-1}$. Therefore, by Lemma $\ref{faCM}$, $(G_n)_{x_1x_2}$ is Cohen-Macaulay with $\alpha((G_n)_{x_1x_2}) = n$.

\textbf {Case 2:} $xy = x_{3k-1}x_{3k}$ for some $k = 1,\ldots,n-1$. Observe that
$$
(G_n)_{x_{3k-1}x_{3k}}=
\begin{cases}
U_1\sqcup \{x_5\} &\text{ if } k = 1, \text{where } U_1\cong G_{n-2} \quad (1)\\
U_2\sqcup \{x_{3n-1}\} & \text{ if } k = n-1, \text{where } U_2\cong G_{n-2} \qquad (2)\\
U_3\sqcup\{x_2\}\sqcup \{x_8\}& \text{ if } k=2, \text{where } U_3\cong G_{n-3} \quad (3)\\
M\sqcup N\sqcup  \{x_{3k+2} \}& \text{ if } 3\leq k< n-1, \quad \qquad \qquad (4)\\
\end{cases}
$$
where $M=G_n[\{x_1,\ldots,x_{3k-6},x_{3k-4}\}]$ and $N=G_n[\{x_{3k+4},\ldots,x_{3n-1}\}]$.

In the three cases $(1)-(3)$, using Lemma \ref{faCM} and Lemma \ref{lemCM}, we have  $(G_n)_{x_{3k-1}x_{3k}}$ is always Cohen-Macaulay with $\alpha((G_n)_{x_{3k-1}x_{3k}})=n-1$. In the last case, we define the map $\varphi: V(H_{k-1}) \longrightarrow V(M)$ as follows:

If $k=3$ then $\varphi(x_i) = x_i$ for all $i = 1, 2, 3$ and $\varphi(x_4)=x_5$.

If $k>3$ then $\varphi(x_i) = x_i$ for all $i = 1,\ldots,3k-9$; $\varphi(x_{3k-8}) = x_{3k-6},\varphi(x_{3k-7}) = x_{3k-7},\varphi(x_{3k-6}) = x_{3k-8}$; and  $\varphi(x_{3k-5}) = x_{3k-4}$.\\
Clearly, $\varphi$ is an isomorphism of two graphs $H_{k-1}$ and $M$. Therefore, $M$ must be Cohen-Macaulay with $\alpha(M)=k-1$ by Lemma \ref{faCM}. Similarly, we have a bijection map $\psi: V(G_{n-k-1}) \longrightarrow V(N)$ is defined by $\psi(x_i)=x_{3k+3+i}$ for all $i=1\ldots,3n-3k-4$. So $G_{n-k-1}\cong N$. Using again Lemma $\ref{faCM}$, $N$ is Cohen-Macaulay with $\alpha(N)=n-k-1$. Thus, $(G_n)_{x_{3k-1}x_{3k}}$ is Cohen-Macaulay with
$$\alpha((G_n)_{x_{3k-1}x_{3k}}) = (k-1) +  (n-k-1)+1 = n-1.$$

\noindent By the same argument, we will obtain the following.

\textbf {Case 3:} $xy = x_{3k}x_{3k+1}$ for some $k=1,\ldots,n-1$. Then,
$$
(G_n)_{x_{3k}x_{3k+1}}\cong
\begin{cases}
H_{n-2}\sqcup \{x_1\}&\text{ if } k = 1\\
G_{k-1} \sqcup G_{n-k-1}\sqcup\{x_{3k-2}\}& \text{ if } k\geqslant 2.
\end{cases}
$$

\textbf {Case 4:} $xy = x_{3k+1}x_{3k+2}$ for some $k=1,\ldots,n-1$. Then,
$$ (G_n)_{x_{3k+1}x_{3k+2}}\cong H_{k-1} \sqcup H_{n-k-1}\sqcup\{x_{3k-1}\}.$$

\textbf {Case 5:} $xy = x_{3k+2} x_{3k-2}$ for some $k=1,\ldots,n-1$. Then,
$$(G_n)_{x_{3k+2}x_{3k-2}}\cong G_{k-1} \sqcup G_{n-k-1}\sqcup\{x_{3k}\}.$$

\textbf {Case 6:} $xy = x_{3k} x_{3k-3}$ for some $2\leq k\leq n-1$. Then,
$$(G_n)_{x_{3k+2}x_{3k-2}}\cong G_{k-1}\sqcup G_{n-k-2}\sqcup\{x_{3k-5}\}\sqcup\{x_{3k+2}\}.$$

\noindent From six cases above, for every edge $xy\in E(G_n)$, we always obtain that $(G_n)_{xy}$ is Cohen-Macaulay with $\alpha((G_n)_{xy}) = n-1$ which completes the proof.
\end{proof}

We are ready to prove the main result of this section.

\begin{thm} \label{T4} Let $G$ be a connected planar graph of girth $4$. Then, $G$ is Gorenstein if and only if $G$ is in the family $\mathcal G$.
\end{thm}

\begin{proof} Note that $G$ has no isolated vertices. Assume that $G$ is Gorenstein, then $G$ is in the class $W_2$ by Lemma $\ref{G}$. Hence, $G$ is in the family $\mathcal G$ by Lemma \ref{P}.

Conversely, if $G$ is a member of $\mathcal G$, we may assume that $G=G_n$ for some $n\geq 3$. By Proposition $\ref{EG1}$, we have $I(G_n)^2$ is Cohen-Macaulay over any field $k$. This fact, together with \cite[Theorem 2.1]{RTY}, implies that $G_n$ is Gorenstein, as required.
\end{proof}

\subsection*{Acknowledgment} We would like to thank Professors L. T. Hoa, R. Woodroofe and S. A. Seyed Fakhari  for helpful comments. We also thank to the referee for their very useful corrections and suggestions. A part of this work was carried out while the second and the third authors visited Genoa University under the support from EMMA in the framework of the EU Erasmus Mundus Action $2$; we would like to thank Professor A. Conca for support and hospitality. The first author is partially supported by the National Foundation for Science and Technology Development (Vietnam) under grant number 101.01-2012.18.


\end{document}